\UseRawInputEncoding
\documentclass[12pt, reqno]{amsart}
\usepackage[margin=1in]{geometry}
\usepackage{amssymb,latexsym,amsmath,amscd,amsfonts}
\usepackage{latexsym}
\usepackage[mathscr]{eucal}
\usepackage{bm}
\usepackage{mathptmx}
\usepackage{amssymb}
\usepackage{amsthm}
\usepackage{dcolumn}
\usepackage[all]{xy}
\usepackage{enumitem}
\usepackage[utf8]{inputenc}

\def\textmatrix#1&#2\\#3&#4\\{\bigl({#1 \atop #3}\ {#2 \atop #4}\bigr)}
\def\dispmatrix#1&#2\\#3&#4\\{\left({#1 \atop #3}\ {#2 \atop #4}\right)}
\newcommand{\beg}{\begin{equation}}
	\newcommand{\eeg}{\end{equation}}
\newcommand{\ben}{\begin{eqnarray*}}
	\newcommand{\een}{\end{eqnarray*}}

\newtheorem{thm}{Theorem}[section]
\newtheorem{cor}[thm]{Corollary}
\newtheorem{lem}[thm]{Lemma}

\newtheorem{prop}[thm]{Proposition}
\numberwithin{equation}{section} \theoremstyle{definition}
\newtheorem{defn}[thm]{Definition}
\newtheorem{rem}[thm]{Remark}

\newtheorem{eg}[thm]{Example}

\newcommand{\HS}{\mathcal H}
\newcommand{\C}{\mathbb{C}}

\newcommand{\D}{\mathbb{D}}
\newcommand{\T}{\mathbb{T}}

\newcommand{\Arm}{\mathbb A_r^m}
\newcommand{\Ar}{\mathbb A_r}

\newcommand{\ov}{\overline}

\def\textmatrix#1&#2\\#3&#4\\{\bigl({#1 \atop #3}\ {#2 \atop #4}\bigr)}
\def\dispmatrix#1&#2\\#3&#4\\{\left({#1 \atop #3}\ {#2 \atop #4}\right)}

\begin{document}
	\title[Dilation of normal operators in an annulus]{Dilation of normal operators associated with an annulus}
	\author[Pal and Tomar]{SOURAV PAL AND NITIN TOMAR}
	
	\address[Sourav Pal]{Mathematics Department, Indian Institute of Technology Bombay,
		Powai, Mumbai - 400076, India.} \email{sourav@math.iitb.ac.in , souravmaths@gmail.com}
	
	\address[Nitin Tomar]{Mathematics Department, Indian Institute of Technology Bombay, Powai, Mumbai-400076, India.} \email{tnitin@math.iitb.ac.in}		
	
	
	
	\thanks{The first named author is supported by the ``Early Research Achiever Award Grant" of IIT Bombay with Grant No. RI/0220-10001427-001, the CDPA of the Government of India. The second named author is supported by the Prime Minister Research Fellowship with Award No. 1300140 of the Government of India.}	
	

	\begin{abstract}
For $0<r<1$, let us consider the following annulus:
\[
\mathbb A_r= \{ z\in \mathbb C\, : \, r<|z|<1 \}.
\]	
A Hilbert space operator $T$ for which $\overline{\mathbb A}_r$ is a spectral set is called an $\mathbb A_r$-\textit{contraction}. Also, a normal operator $U$ whose spectrum lies on the boundary $\partial \mathbb A_r$ of $\mathbb A_r$ is called an $\mathbb A_r$-\textit{unitary}. We prove that any $m$ number of commuting normal $\mathbb A_r$-contractions $N_1, \dots , N_m$ can be simultaneously dilated to commuting $\mathbb A_r$-unitaries $U_1, \dots , U_m$. To construct such a dilation, we solve a Dirichlet problem for the polyannulus $\mathbb A_r^m$. Also, we show that any finitely many doubly commuting subnormal $\mathbb A_r$-contractions simultaneously dilate to commuting $\mathbb A_r$-unitaries. Finally, we show that such a simultaneous $\mathbb A_r$-unitary dilation holds for any finite number of doubly commuting $2 \times 2$ scalar $\mathbb A_r$-contractions.

	\end{abstract}

	\maketitle
	
	\section{Introduction}
	
	\vspace{0.3cm}
	
	\noindent  \noindent Throughout the paper every operator is a bounded linear operator acting on a complex Hilbert space. For a Hilbert space $\HS$, $\mathcal B(\HS)$ is the space of all operators defined on $\HS$. A contraction is an operator whose norm is less than or equal to $1$. We denote by $\C , \mathbb R, \mathbb Z, \mathbb N$ the set of complex numbers, real numbers, integers and positive integers respectively. Also, $\D$ and $\T$ denote the open unit disc and the unit circle respectively in the complex plane $\C$. We define Taylor joint spectrum, spectral set, rational dilation etc. in Section 2.\\ 

Every operator is a scalar multiple of a contraction and a contraction can be realized as a part of a unitary operator as was shown by Sz.-Nagy in \cite{nagy1}.

\begin{thm}[Sz.-Nagy, 1953]
If $T$ is a contraction acting on a Hilbert space $\mathcal H$,
then there exists a Hilbert space $\mathcal K \supseteq \mathcal
H$ and a unitary $U$ on $\mathcal K$ such that
\[
f(T)=P_{\mathcal H}f(U)|_{\mathcal H}, \quad \text{ for every rational function } f \text{ with poles off } \; \ov{\D},
\]
where $f(T)=p(T)q(T)^{-1}$ when $f=p \slash q$ with $p, q \in \C[z]$.

\end{thm}
In operator theoretic language this is called a rational dilation of a contraction (see Section 2 for details). While Sz.-Nagy characterized a contraction as a compression of a unitary operator, von Neumann \cite{von-Neumann} described a contraction as an operator having the closed unit disk $\ov{\D}$ as a spectral set.
\begin{thm}[von Neumann, 1951]
An operator $T$ is a contraction if and only if $\overline{\mathbb
D}$ is a spectral set for $T$.
\end{thm}
Needless to mention that a unitary is a normal operator having its spectrum on $\T$. So, Sz.-Nagy's theorem roughly states that an operator that lives inside $\ov{\D}$ can be dilated to a normal operator lying on the boundary of $\ov{\D}$. These compelling results insist numerous mathematicians to study operators having a compact set as a spectral set and their normal dilation to the boundary, e.g. see classic \cite{NagyFoias, Paulsen} and the references therein. The notion of spectral set for an operator was generalized later by Arveson (see \cite{ArvesonI}) for a finite tuple of commuting operators and consequently we witness interesting interplay between structural properties of an operator tuple and the complex geometry of an underlying compact subset of $\mathbb C^m$ associated with the tuple (e.g. \cite{ AgMcC, AM05, pal-shalit}). Note that in multivariable setting the distinguished boundary of a compact set $X$ is considered in place of the topological boundary (see Section 2 for details). After the discovery of unitary dilation of a contraction, we are left with the following problem which remains open till date: if a compact set $X \subset \C^m$ is a spectral set for a commuting operator tuple $(T_1, \dots , T_m)$ on a Hilbert space $\HS$, then whether or not one can find a commuting normal operator tuple $(N_1, \dots , N_m)$ acting on a Hilbert space $\mathcal K \supseteq \HS$ and having its joint spectrum on the distinguished boundary of $X$ such that
\[
f(T_1, \dots , T_m) = P_{\HS} f(N_1, \dots , N_m)|_{\HS}
\] 
for every rational function $f=p \slash q$, where $p, q \in \C[z_1, \dots , z_m]$ with $q$ having no zeros in $X$. Apart from the unit disk $\D$, this rational dilation problem has affirmative answers for an annulus \cite{Agler}, for the bidisc $\D^2$ \cite{ando} and the symmetrized bidisk $\mathbb G_2$ \cite{ay-jfa, tirtha-sourav}, where
\[
\mathbb G_2=\{ (z_1+z_2,z_1z_2)\,:\, z_1,z_2 \in \D \}.
\]
Also, we have failure of rational dilation on a triply connected domain \cite{DM, ahr}, on tridisk $\D^3$ \cite{Parrott} etc. However, no success of rational dilation is yet known for a domain in more than two variables. Indeed, finding dilation in multivariable setting has always been a challenging and difficult problem. Ever since Agler profoundly established the success of rational dilation on an annulus, a canonical question was triggered: what happens with a polyannulus ? The aim of this article is to make a first step towards this problem by finding normal boundary dilation for a commuting tuple of normal operators associated with a closed annulus.\\ 

For $0<r<1$, consider the annulus $\mathbb A_r=\{z\in \C \,: r<|z|<1 \}$. Thus, an $m$-\textit{annulus} or simply a \textit{polyannulus} is the following domain:
\[
\Arm = \{ (z_1, \dots , z_m)\in \C^m\,:\, z_i \in \Ar\,, 1\leq i \leq m \}.
\]
Evidently the topological boundary $\partial \Ar$ of $\Ar$ is a union of $\T$ and $r\T$. We shall prove in Section \ref{basic} that the distinguished boundary of $\Arm$ is equal to $(\partial \Ar)^m$, which is analogous to $\T^m$ being the distinguished boundary of the polydisc $\D^m$. After Agler's famous article \cite{Agler}, there have been quite a few remarkable works on function and operator theory on an annulus, e.g. \cite{Ab1, Ab2, Ball, Dmitry, P-SM, N-S1, Sarason, Tsikalas}.

\begin{defn}
For $m\geq 1$, a commuting tuple of operators $(T_1, \dots , T_m)$ is said to be an $\Arm$-\textit{contraction} if $\ov{\mathbb A}_r^m$ is a spectral set for $(T_1, \dots , T_m)$. An $\Arm$-contraction $(T_1, \dots , T_m)$ is called \textit{normal} if each $T_i$ is a normal operator. Also, a commuting normal tuple of operators $(U_1, \dots , U_m)$ is said to be an $\Arm$-\textit{unitary} if the Taylor joint spectrum $\sigma_T(U_1, \dots , U_m)$ lies on the distinguished boundary $(\partial \Ar)^m$ of $\Arm$.
\end{defn}

It follows from the properties of Taylor joint spectrum that if $N_1, \dots , N_m$ are commuting normal operators with $\ov{\mathbb A}_r$ being a spectral set for each $N_i$, then $\ov{\mathbb A}_r^m$ is a spectral set for the tuple $(N_1, \dots , N_m)$ and vice-versa. To be more specific $(U_1, \dots , U_m)$ is an $\Arm$-unitary if and only if each $U_i$ is an $\Ar$-unitary. Thus, finding a simultaneous $\Ar$-unitary dilation for commuting normal $\Ar$-contractions $N_1, \dots , N_m$ is equivalent to seeking an $\Arm$-unitary dilation for the normal $\Arm$-contraction $(N_1, \dots , N_m)$. In Theorem \ref{main}, the main result of this paper, we show that every normal $\Arm$-contraction admits an $\Arm$-unitary dilation. In Theorem \ref{subnormal}, we make a refinement of this result by showing that every subnormal $\Arm$-contraction possesses an $\Arm$-unitary dilation. A next step to dilating commuting normal $\Ar$-contractions $N_1, \dots , N_m$ is to find simultaneous $\Ar$-unitary dilation for doubly commuting $\Ar$-contractions. A commuting tuple of operators $(A_1, \dots , A_m)$ is said to be \textit{doubly commuting} if $A_iA_j^*=A_j^*A_i$ for $1\leq i,j \leq m$ with $i \neq j$. In Theorem \ref{subnormal1}, we show that any doubly commuting subnormal $\Ar$-contractions $S_1, \dots , S_m$ can have simultaneous commuting $\Ar$-unitary dilation. Also, in Theorem \ref{DC_2} we show that any $m$ number of doubly commuting $2\times 2$ scalar $\Ar$-contractions simultaneously dilate to commuting $\Ar$-unitaries.\\

In Theorem \ref{Dirichlet}, we solve a Dirichlet problem for the polyannulus $\Arm$ which actually plays the central role in the proof of the main dilation theorem (Theorem \ref{main}). Indeed, for a normal $\Arm$-contraction $(N_1, \dots , N_m)$, we have by spectral theorem a spectral measure which is supported at $\ov{\mathbb A}_r^m$. Now if we want to dilate such a tuple to an $\Arm$-unitary, then we first need to construct a spectral measure that is supported at the distinguished boundary $b\Arm$. The Dirichlet problem for the polyannulus (i.e. Theorem \ref{Dirichlet}) precisely solves the purpose.

In Section \ref{Analytic}, we prove a few results associated with rational functional calculus for the polyannulus $\Arm$. These results are used in sequel.

\vspace{0.3cm}

	\section{Preliminaries}\label{basic}
	
\vspace{0.3cm}

\noindent In this Section, we recall from the literature a few basic facts and some relevant existing results which will be used in sequel. We begin with the definition of the Taylor joint sectrum of a tuple of commuting operators.

	\subsection{The Taylor joint spectrum}  Let $\Gamma$ be the exterior algebra on $n$ generators $e_1, e_2, \dotsc, e_n$ with identity $e_0 \equiv 1. \Lambda$ is the algebra of the forms $e_1, e_2, \dotsc, e_n$ with complex coefficients, subject to the property $e_ie_j+e_je_1=0 (1 \leq i,j \leq n).$ Let $E_i: \Lambda \to \Lambda$ denote the operator, given by $E_i\xi=\xi (\xi \in \Lambda, 1 \leq i,j \leq n).$ If we declare $\{e_{i_1}, \dotsc, e_{i_k} : \ 1 \leq i_1 < \dotsc < i_k \leq n \}$ to be an orthogonal basis, the exterior algebra $\Gamma$ becomes a Hilbert space, admitting an orthogonal decomposition $\Lambda=\overset{n}{\underset{k=1}{\oplus}} \Lambda^k$ where dim$\Lambda^k=\begin{pmatrix}
		n\\ k\\
	\end{pmatrix}.$ Thus, each $\xi \in \Lambda$ admits a unique orthogonal decomposition  $\xi=e_1\xi'+\xi''$ have no $e_i$ contribution. It then follows that $E_i^*\xi=\xi',$ and we have that each $E_i$ is a partial isometry, satisying $E_i^*E_j+E_j^*E_i=\delta_{ij}.$ Let $\mathscr{X}$ be a normed space and $\underline{T}=(T_1, \dotsc, T_n)$ be a commuting $n$-tuple of bounded operators on $\mathscr{X}$ and set $\Lambda(\mathscr{X})=\mathscr{X}\otimes_\mathbb{C}\Lambda.$ We define $D_{\underline{T}}: \Lambda(\mathscr{X}) \to \Lambda(\mathscr{X})$ by
	\begin{center}
		$D_{\underline{T}}=\overset{n}{\underset{i=1}{\sum}}T_i \otimes E_i.$
	\end{center}
	Then it is easy to see that $D_{\underline{T}}^2=0,$ so $RanD_{\underline{T}} \subset KerD_{\underline{T}}.$ The commuting $n$-tuple is said to be \textit{non-singular} on $\mathscr{X}$ if $RanD_{\underline{T}} = KerD_{\underline{T}}. $
	
	\begin{defn}
		The Taylor joint spectrum of $\underline{T}$ on $\mathscr{X}$ is the set 
		\begin{center}
			$\sigma_T(\underline{T}, \mathscr{X})=\{\lambda=(\lambda_1, \dotsc, \lambda_n) : \ \underline{T}-\lambda \ \mbox{is singular}\}.$
		\end{center}
		$ $
		For a further reading on Taylor joint spectrum, reader is referred to Taylor's works \cite{TaylorI} and \cite{TaylorII}.
	\end{defn}
	\subsection{The distinguished boundary} For a compact subset $X$ of $\C^m$, let $\mathscr{A}(X)$ be the algebra of continuous complex-valued funtions on $X$ that are holomorphic in the interior of $X$. A \textit{boundary} for $X$ is a closed subset
	$C$ of $X$ such that every function in $\mathscr{A}(X)$ attains its maximum modulus on $C$. It follows from the theory of uniform algebras that the intersection of all the boundaries of $X$ is also a boundary of $X$ and it is the smallest among all boundaries. This is called the \textit{distinguished boundary} of $X$ and is denoted by $bX$. For a bounded domain $G \subset \C^m$, we denote by $bG$ the distinguished boundary of $\ov{G}$ and for the sake of simplicity we call it the distinguished boundary of $G$.
	
	\subsection{Spectral set, complete spectral set, rational dilation and Arveson's theorem}
	Let $X$ be a compact subset of $\mathbb{C}^m$ and $Rat(X)$ be the algebra of rational functions $p\slash q$, where $p,q \in \C[z_1, \dots , z_m]$ such that $q$ does not have any zeros in $X$. Let $\underline{T}=(T_1, \dotsc, T_m)$ be a commuting tuple of operators acting on a Hilbert space $\mathcal{H}$. Then $X$ is said to be a \textit{spectral set} for $\underline{T}$ if the Taylor joint spectrum of $\underline{T}$ is contained in $X$ and von Neumann's inequality holds for any $f \in Rat(X)$, i.e. 
	\[
		\|f(\underline{T})\| \leq \sup_{\xi \in X} \;|f(\xi)|=\|f\|_{\infty, \, X},
	\]
	where $f(\underline{T})=p(\underline{T})q(\underline{T})^{-1}$ when $f=p \slash q$. Also, $X$ is said to be a \textit{complete spectral set} if for any $F=[f_{ij}]_{n\times n}$, where each $f_{ij}\in Rat\,(X)$, we have
	\[
		\|f(\underline{T})\|=\|\,[f_{ij}(\underline{T})]_{n\times n} \,\| \leq \sup_{\xi \in X} \;\| \, [f_{ij}(\xi)]_{n\times n} \,\|.
	\]
	
	\noindent A commuting $m$-tuple of operators $\underline{T}$ having $X$ as a spectral set, is said to have a \textit{rational dilation} or \textit{normal b$X$-dilation} if there exist a Hilbert space $\mathcal{K}$, an isometry $V: \mathcal{H} \to \mathcal{K}$ and a commuting $m$-tuple of normal operators $\underline{N}=(N_1, \dots , N_m)$ on	$\mathcal{K}$ with $\sigma_T(\underline{N}) \subseteq bX$ such that 
	\begin{center}
		$f(\underline{T})=V^*f(\underline{N})V \ \ \ \text{ for all }\; f \in Rat(X)$ \,.
	\end{center}
In other words, $f(\underline{T})=P_\mathcal{H}f(\underline{N})|\mathcal{H}$ for every $f \in Rat(X)$ when $\mathcal{H}$ is realized as a closed subspace of $\mathcal{K}$.

The following celebrated theorem due to Arveson combines these two concepts.

\begin{thm} [Arveson, \cite{ArvesonI}] \label{thm:W.Arveson}
 A commuting operator tuple $\underline{T}$ admits a normal $bG$-dilation if and only if $\ov{G}$ is a complete spectral set for $\underline{T}$.
 \end{thm}
	\subsection{Operator-valued measure and Naimark Dilation Theorem}
	Let $X$ be a compact Hausdorff space and $\mathcal{B}$ be the $\sigma$-algebra of Borel sets in $X$.	An operator-valued measure on $X$ is a map $E: \mathcal{B} \rightarrow \mathcal B(\mathcal{H})$ which is weakly countably additive, i.e. if $\left\{B_{i}\right\}$ is a countable collection of disjoint Borel sets whose union is $B$, then
		\[
		\langle E(B) x, y\rangle=\sum_{i}\left\langle E\left(B_{i}\right) x, y\right\rangle \quad \text{ for all } x, y \in \mathcal{H}.
		\]
	 The measure is called \textit{bounded} if
		$
		\sup \{\|E(B)\|: B \in \mathcal{B}\}<\infty.
		$  
		Also, the measure is said to be \textit{regular} if the induced complex measure
		\begin{equation} \label{regular}
			\mu_{x, y}(B)=\langle E(B) x, y\rangle \,, \quad \text{ for all } x,y \in \HS
		\end{equation}
		is regular. Given a regular bounded $\mathcal B(\mathcal{H})$-valued measure $E$, one obtains a bounded linear map $\phi_{E}: C(X) \rightarrow \mathcal B(\mathcal H)$ defined by
	\begin{equation} \label{meausre_linear_map}
	\left\langle\phi_{E}(f) x, y\right\rangle=\underset{X}{\int} f d \mu_{x, y}. 
	\end{equation}
	Conversely, for a bounded linear map $\phi: C(X) \rightarrow \mathcal B(\mathcal{H})$, one can define using (\ref{meausre_linear_map}) a regular Borel measures $\{\mu_{x, y}\}$ for each pair $x,y$ in $\mathcal{H}$. Again, for each Borel set $B$, there is a unique bounded operator $E(B)$, given by $(\ref{regular})$ such that $B \mapsto E(B)$ defines a bounded regular $\mathcal B(\mathcal{H})$-valued measure. Therefore, there is a one-to-one correspondence between the bounded linear maps of $C(X)$ into $\mathcal B(\mathcal{H})$ and the regular bounded $\mathcal B(\mathcal{H})$-valued measures.	
		
A regular $\mathcal B(\mathcal{H})$-valued measure $E$ on $X$ is called a \textit{spectral measure} if for any Borel subsets $S_1, S_2$ of $X$ we have
		\[
		E(S_1 \cap S_2)=E(S_1)E(S_2).
		\]
	Moreover, $E$ is called \textit{positive} if $E(S) \geq 0$ for every Borel set $S$ in  $X$. It is well-known that there is a one-to-one correspondence between the operator measures and linear maps. This leads to the following famous dilation result due to Naimark.
	
	\begin{thm}[\cite{Paulsen}, Theorem 4.6]\label{Naimark}
		Let $E$ be a regular, positive, $\mathcal B(\mathcal{H})$-valued measure on $X.$ Then there exist a Hilbert space $\mathcal{K},$ a bounded linear operator $V:\mathcal{H} \to \mathcal{K}$ and a regular spectral $\mathcal B(\mathcal{K})$-valued measure $F$ on $X$ such that 
		\[
		E(B)=V^*F(B)V.
		\]
		Moreover, if the map $\phi_E$ is unital then $V$ is an isometry.
	\end{thm}
	

	\subsection{Reisz-Kakutani Representation Theorem}
	A central theorem in the classical measure theory is the Reisz-Kakutani representation theorem which relates the linear functionals on the space $C(X)$ with regular Borel measures on $X$, where $X$ is a locally compact Hausdorff space.
	
	\begin{thm} \label{thm:2RK}
		Let $X$ be a compact Hausdorff space. For any continuous linear functional $\psi$ on $C(X),$ there exists a unique regular Borel complex measure $\mu$ on $X$ such that
		\[
		\psi(f)=\underset{X}{\int}f d\mu \ \ (f \in C(X))
		\]
		and $\|\psi\|=\|\mu\|$ which is the total variation of $\mu$. Moreover, $\psi$ is positive if and only if the measure $\mu$ is non-negative. 
	\end{thm}
	
	\vspace{0.3cm}

	\section{The distinguished boundary of the polyannulus $\Arm$ and the polyannulus-unitaries} \label{sec:04}

\vspace{0.4cm}	
	
	 \noindent In this Section, we determine the distinguished boundary of polyannulus $\mathbb{A}_r^m$ and find a characterization for a polyannulus-unitary, i.e. an $\mathbb A_r^m$-unitary. The distinguished boundary of the polydisc $\mathbb D^m$ is the $m$-torus $\T^m$, which is nothing but the Cartesian product of the (topological) boundary of the disc $\D$. We shall see here that an analogue holds for the polyannulus $\Arm$. Indeed, the distinguished boundary of $\Arm$ is $(\partial \Ar)^m$. We will also have an operator theoretic version of this result, i.e. a commuting tuple of operators $(U_1, \dots , U_m)$ is an $\Arm$-unitary if and only if each $U_i$ is an $\Ar$-unitary.

	  If there exist a function $f \in \mathscr{A}(\mathbb{A}_r^m)$ and a point $a \in \overline{\mathbb{A}}_r^m$ such that $f(a)=1$ and $|f(z)|<1$ for all $z \in \overline{\mathbb{A}}_r^m\setminus\{a\},$ then $a$ must belong to $b{\mathbb{A}_r^m}.$ We call such a point $a$ a \textit{peak point} of $\overline{\mathbb{A}}_r^m$ and the function $f$ is called a \textit{peaking function} for the point $a.$ Here we show that $b\Arm$ actually consists of all peak points of $\ov{\Ar}^m$.

	\begin{thm}\label{distt}
		For $a=(a_1,  \dotsc, a_m) \in \mathbb{C}^m,$ the following are equivalent.
		\begin{enumerate}
			\item $a$ is a peak point of $\overline{\mathbb{A}}_r^m;$
			\item $a_k \in \partial{\mathbb{A}_r}$ for all $k=1,2,\dotsc,m;$
			\item $a \in b{\mathbb{A}_r^m},$ the distinguished boundary of $\mathbb{A}_r^m.$
		\end{enumerate}
		Consequently, we have that $b{\mathbb{A}_r^m}= 
		\underbrace{\partial \mathbb{A}_r \times \dotsc \times \partial\mathbb{A}_r}_{m-times} .$	
	\end{thm}
	\begin{proof}
		We shall prove  $(1) \iff (2) \iff (3).$ 
		$ $\\
		\item $(1) \implies (2)$ If  $(a_1, \dotsc, a_m)$ is a peak point for $\mathscr{A}(\mathbb{A}_r^m)$ then there exists $f \in \mathscr{A}(\mathbb{A}_r^m)$ such that
		\[
		f(a_1, \dotsc, a_m)=1 \;\; \& \;\; |f(z_1,\dotsc, z_m)|<1 \;\;
		 \text{ for all } \;\; (z_1,\dotsc, z_m) \in \overline{\mathbb{A}}_r^m \setminus \{ (a_1, \dotsc, a_m)\}.
		\]
		Define $g: \overline{\mathbb{A}}_r \to \mathbb{C}$ by $g(z)=f(z,a_2, \dotsc, a_m)$. Since $f \in \mathscr{A}(\mathbb{A}_r^m)$, $g$ is analytic in $\mathbb{A}_r$ and continuous on $\overline{\mathbb{A}}_r$. Now $g$ attains its maximum modulus at $a_1.$ Thus, by the Maximum principle we have that $a_1 \in \partial\mathbb{A}_r$. Similarly, one can show that  $a_k \in \partial\mathbb{A}_r$ for all $k=1,\dotsc,m$. \\
		\item $(2) \implies (1)$ For $\theta, \phi \in [0, 2\pi],$ the functions 
		$f_\theta, g_\phi: \overline{\mathbb{A}}_r \to \mathbb{C}$ defined by
		\begin{equation*}
			f_\theta(z)=\frac{e^{i\theta}}{2e^{i\theta}-z} \quad \mbox{and} \quad g_\phi(z)=\frac{re^{i\phi}}{2z-re^{i\phi}}
		\end{equation*}
		are analytic in $\mathbb{A}_r$ and continuous on $\overline{\mathbb{A}}_r$. Since $f_\theta(e^{i\theta})=1$, we have that $|z-2e^{i\theta}|>1$ for any $z\ne e^{i\theta}$ in $\overline{\mathbb{A}}_r$ which implies that 
		\begin{equation*}
			|f_\theta(z)|=\frac{1}{|z-2e^{i\theta}|}<1.
					\end{equation*}
		Hence, $e^{i\theta}$ is a peak point of $\overline{\mathbb{A}}_r$. Again, $g_\phi(re^{i\phi})=1$ and for any $z \in \overline{\mathbb{A}}_r \setminus \{re^{i\phi}\}$, we have $|z-{re^{i\phi}}/{2}|> {r}/{2}$, i.e. $|2z-re^{i\phi}| >r$ and thus 
		\begin{equation*}
			 |g_\phi(z)|=\frac{r}{|2z-re^{i\phi}|}<1.
		\end{equation*}
		Therefore, $re^{i\phi}$ is a peak point of $\overline{\mathbb{A}}_r$. If $(a_1, \dotsc, a_m) \in \partial \mathbb{A}_r \times \dotsc \times \partial \mathbb{A}_r$ then each $a_k$ is either $e^{i\theta_k}$ or $re^{i\phi_k}$, for some $\theta_k, \phi_k \in [0, 2\pi]$. Define $h:\overline{\mathbb{A}}_r^m \to \mathbb{C}$ by 
		\begin{equation*}
			h(z_1,\dotsc, z_m)=\overset{m}{\underset{k=1}{\prod}}f_{\theta_k}^{m_k}(z_k)g_{\phi_k}^{(1-m_k)}(z_k), \quad \text{where} \quad m_k=  \left\{
		\begin{array}{ll}
			1, & a_k=e^{i\theta_k} \\
			0, & a_k=re^{i\phi_k}. \\
		\end{array} 
		\right.
		\end{equation*}
		Then $h$ is in $\mathscr{A}(\mathbb{A}_r^m)$ as each $f_{\theta_k}$ and $g_{\phi_k}$ are in $\mathscr{A}(\mathbb{A}_r).$ Since $f_{\theta_k}$ and $g_{\phi_k}$ peaks at $e^{i\theta_k}$ and $re^{i\phi_k}$ respectively, we have that $h(a_1, \dotsc, a_m)=1$ and $|h(z_1,\dotsc, z_m)|<1$ for all $(z_1,\dotsc, z_m)$ in $\overline{\mathbb{A}}_r^m \setminus \{(a_1, \dotsc, a_m)\}$. Hence, $(a_1, \dotsc, a_m)$ is a peak point for $\mathscr{A}(\mathbb{A}_r^m)$ with $h$ being a peaking function.\\
		\item $(2) \iff (3)$. Note that $\mathscr{A}(\mathbb{A}_r^m)$ is a uniform function algebra on a compact metrizable space $\overline{\mathbb{A}}_r^m$. It was proved by Dales in \cite{G.Dales} that if $\mathcal A$ is a Banach function algebra on a compact metrizable space $X$, then the set of peak points of $\mathcal A$ is dense in the Shilov boundary of $\mathcal A$ (with respect to $X$). So, it follows that the distinguished boundary of $\mathbb{A}_r^m$ is given by the closure of the peak points of $\overline{\mathbb{A}}_r^m$. In this case, the set of peak points is $\partial\mathbb{A}_r \times \dotsc \times \partial\mathbb{A}_r$ which is closed. Hence, we have that $b\mathbb{A}_r^m=\underbrace{\partial\mathbb{A}_r \times \dotsc \times \partial\mathbb{A}_r}_{m-times}.$ The proof is now complete.
		
	\end{proof}
	
	\noindent Recall that an operator $T \in \mathcal{B}(\mathcal{H})$ is an  $\mathbb{A}_r$-\textit{unitary} if $T$ is normal and $\sigma(T) \subseteq \partial \mathbb{A}_r$. Here we characterize a commuting normal operator tuple $(U_1, \dots , U_m)$ having its Taylor joint spectrum in the distinguished boundary $b\Arm$ of $\Arm$. Such an operator tuple is called an $\Arm$-unitary. First let us note down a few important observations about an $\Ar$-unitary.
	
	\begin{thm}[\cite{N-S1}, Theorem 3.2]
	An operator $N \in \mathcal B(\mathcal K)$ is an $\Ar$-unitary if and only if there is an orthogonal decomposition $\mathcal K = \mathcal K_1 \oplus \mathcal K_2$ and there are unitaries $U_1 \in \mathcal B(\mathcal K_1)$ and $U_2 \in \mathcal B(\mathcal K_2)$ such that	
	\[
       N=	\begin{bmatrix}
            U_1 &  0     \\
            0   &  rU_2  \\ 		
	        \end{bmatrix}
	   \quad \mbox{ with respect to } \; \mathcal K =\ \mathcal K_1 \oplus \mathcal K_2 .    
\]
\end{thm}

	\begin{lem}
		\label{T,rT} Let $T \in \mathcal{B}(\mathcal{H}).$ Then $T$ is an $\mathbb{A}_r$-unitary if and only if $rT^{-1}$ is an $\mathbb{A}_r$- unitary.
	\end{lem}
	\begin{proof} $T$ is normal if and only if $rT^{-1}$ is normal and
		\begin{equation*}
			\begin{split}
				\sigma(rT^{-1}) \subseteq \mathbb{T} \cup r\mathbb{T}  \iff \sigma(T^{-1}) \subseteq r^{-1}\mathbb{T} \cup \mathbb{T}
				\iff \sigma(T) \subseteq r\mathbb{T} \cup \mathbb{T}.
			\end{split}
		\end{equation*}
		The last equivalence holds from the spectral mapping theorem for the continuous functional calculus of a normal operator from which we have that $\sigma(T^{-1})=\{z^{-1}| z \in \sigma(T)\}.$
	\end{proof}
	\begin{lem}\label{Polyannulus-each annulus}
		An $m$-tuple $\underline{T}=(T_1, \dotsc, T_m)$ of commuting normal operators in $\mathcal{B}(\mathcal{H})$ is $\mathbb{A}_r^m$-unitary if and only if each $T_k$ is an $\mathbb{A}_r$-unitary.
	\end{lem}
	\begin{proof} Assume that $\underline{T}=(T_1, \dotsc, T_m)$ is an $\Arm$-unitary. Then $T_1, \dots , T_m$ are commuting normal operators with $\sigma_T(\underline{T}) \subseteq b \mathbb{A}_r^m$. Then the projection property of the Taylor joint spectrum implies that $ \sigma(T_k) \subseteq \partial \mathbb{A}_r$. Hence each $T_k$ is an $\mathbb{A}_r$-unitary. Conversely, suppose each $T_k$ is an $\mathbb{A}_r$-unitary. Then
	 \[
	 \sigma_T(\underline{T}) \subseteq \sigma(T_1) \times \dotsc \times \sigma(T_m) \subseteq \underbrace{\partial \mathbb{A}_r \times \dots \times \partial \mathbb{A}_r}_{m-times}=b\Arm
	 \]
	 and thus $(T_1, \dots , T_m)$ is an $\Arm$-unitary.
	 
	\end{proof}
	

	\section{Rational functional calculus on the polyannulus $\Arm$} \label{Analytic}
	
\vspace{0.3cm}	
	
	\noindent Recall that an $\Arm$-contraction $\underline{T}=(T_1, \dots , T_m)$ admits a rational dilation or a normal $b\Arm$-dilation if there is a commuting normal $m$-tuple of operators $\underline{N}=(N_1, \dots , N_m)$ on a bigger Hilbert space $\mathcal K \supseteq \HS$ with $\sigma_T(\underline{N})\subseteq \ov{\mathbb A}_r^m$ such that $f(\underline{T})=P_{\HS} f(\underline{N})|_{\HS}$ for any rational function $f=p \slash q$, where $p,q \in \C [z_1, \dots , z_m]$ with $q$ having no zeros in $\ov{\mathbb A}_r^m$. In this Section, we show that it suffices to take into consideration a smaller class of functions, namely the monomials with both positive and negative integral powers. Before that we briefly recollect a few basic facts from the literature \cite{Range} about multiple series.
	
	\begin{defn}
		The multiple series $\underset{k \in \mathbb{N}^m}{\sum}c_k$ is called \textit{convergent} if 
		\[
		\underset{k \in \mathbb{N}^m}{\sum}|c_k|=\sup\bigg\{\underset{k \in P}{\sum}|c_k|: \ P \ \mbox{finite}
		\bigg\} < \infty.
		\]
	\end{defn}
	
	Below we have a characterization for convergence of a multiple series.
	
	\begin{prop}
		The multiple series $\underset{k \in \mathbb{N}^m}{\sum}c_k$ is convergent if and only if  given any bijection $\nu: \mathbb{N} \to \mathbb{N}^m,$ the ordinary series
		$ 
		\overset{\infty}{{\underset{j=0}{\sum}}}c_{\nu(j)}
		$
		converges in the usual sense to a limit $L \in \mathbb{C}$ which is independent of $\nu.$
	\end{prop}
	
	For further details of the above result see Section 1.5 in \cite{Range}. We also have the following Laurent series expansion in the multivariable setting. 
	
	\begin{prop}[\cite{Range}, Theorem 1.4] \label{prop:503}
	
	Every function analytic on $\overline{\mathbb{A}}_r^m$ has a unique representation
		\[
		f(z)=\underset{k \in \mathbb{Z}^m}{\sum}f_kz^k \ \ \mbox{for} \ z \in \overline{\mathbb{A}}_r^m
		\]
		and the series converges uniformly and absolutely on $\overline{\mathbb{A}}_r^m.$
		
	\end{prop}
	
	\noindent Following the above series expansion of a rational function $f$ on $\overline{\mathbb{A}}_r^m$, we now show how $f(\underline{T})$ makes sense when $\underline{T}=(T_1, \dots , T_m)$ is an $\Arm$-contraction. We need the following two lemmas for this purpose.
	
	\begin{lem}[\cite{N-S1}, Lemma 2.2]  \label{||T||_in_Ar}
		For every $\mathbb{A}_r$-contraction $T$ on a Hilbert space $\mathcal{H},$ we have that $r \leq  \|T\| \leq 1.$
	\end{lem}

	\begin{lem} [\cite{N-S1}, Lemma 2.5] \label{T and rT^-1}
		An operator $T$ is an $\mathbb{A}_r$-contraction if and only if $rT^{-1}$ is an $\mathbb{A}_r$-contraction.
	\end{lem}
	
	\begin{prop}\label{Functional calculus}
		For an $\Arm$-contraction $\underline{T}=(T_1, \dots , T_m)$ on a Hilbert space $\mathcal{H}$ and a rational function $f(z)=\underset{k \in \mathbb{Z}^m}{\sum}f_kz^k$ on $\overline{\mathbb{A}}_r^m$, the series $\underset{k \in \mathbb{Z}^m}{\sum}f_k\underline{T}^k$ converges in operator norm and 
		\[
		f(\underline{T})=\underset{k \in \mathbb{Z}^m}{\sum}f_k\underline{T}^k.
		\]
	\end{prop}
	\begin{proof}
		From Lemma \ref{||T||_in_Ar} and Lemma \ref{T and rT^-1}, we have that $r\leq \|T_j \|\leq 1$ and $r \leq \|rT_j^{1}\| \leq 1$ for each $j=1, \dots, m.$ We rewrite the series as
		\[
		\underset{k \in \mathbb{Z}^m}{\sum}f_kz^k=\underset{k \in \mathbb{Z}^m}{\sum}\Tilde{f}_k\Tilde{z_1}^{|k_1|} \dotsc \Tilde{z_m}^{|k^m|}, \quad \text{where } \;\; \Tilde{z_j}^{|k_j|}= \left\{
		\begin{array}{ll}
			z_j^{k_j} &  \ k_j \geq 1 \\
			(rz_j^{-1})^{-k_j} & \; k_j <0 \; ,\\
		\end{array}
		\right.
		\]
		and $\Tilde{f_k}=(r^{k_{j_1}}\dots r^{k_{j_l}})f_k, \;\;\;  l \;$  being the number of negative integers in the tuple $k=(k_1, k_2, \dotsc, k_m)$. In a similar fashion let us write
		\begin{equation*}
			\begin{split}
				\underset{k \in \mathbb{Z}^m}{\sum}f_k\underline{T}^k =\underset{k \in \mathbb{Z}^m}{\sum}\Tilde{f}_k\Tilde{T_1}^{|k_1|} \dotsc \Tilde{T_m}^{|k^m|} \; , \quad \text{where }\;\; \Tilde{T_j}^{|k_j|}= \left\{
		\begin{array}{ll}
			T_j^{k_j} &  \ k_j \geq 1 \\
			(rT_j^{-1})^{-k_j} &  \; k_j <0 \;.\\
		\end{array} \ \ 
		\right.
			\end{split}
		\end{equation*}
		For each $j=1,2,\dotsc, m$, we have that $r\leq \|\Tilde{T_j}\| \leq 1$ and therefore we have that
		\begin{equation*}
			\begin{split}
				\underset{k \in \mathbb{Z}^m}{\sum} \|f_k\underline{T}^k \| 
				&=\underset{k \in \mathbb{Z}^m}{\sum}\bigg|\bigg|\Tilde{f}_k\Tilde{T_1}^{|k_1|}\dotsc \Tilde{T_m}^{|k^m|}\bigg|\bigg|\\ & \leq   \underset{k \in \mathbb{Z}^m}{\sum}|\Tilde{f}_k| \ \|\Tilde{T_1}\|^{|k_1|} \dotsc \|\Tilde{T_m}\|^{|k^m|} \\
			\end{split}
		\end{equation*}
		which is a convergent series as $\|\Tilde{T}_j\| \in \overline{\mathbb{A}}_r$ and the series of $f$ converges absolutely in $\overline{\mathbb{A}}_r^m$. Therefore, the series $\underset{k \in \mathbb{Z}^m}{\sum}f_k\underline{T}^k$ converges in operator norm to a bounded operator say $\widetilde T$. Now the continuity property of the holomorphic functional calculus implies that $f(\underline{T}) = \underset{k \in \mathbb{Z}^m}{\sum}f_k\underline{T}^k$ and the proof is complete.
		
	\end{proof}
	
In light of Propositions \ref{prop:503} \& \ref{Functional calculus}, the definition of normal $b\Arm$-dilation reduces to considering only monomials of positive and negative integral powers in the following way.
	
	\begin{prop}\label{Power}
	
		Let $\underline{T}=(T_1, \dotsc, T_m)$ be an $\mathbb{A}_r^m$-contraction on $\mathcal{H}$ and $\underline{U}=(U_1, \dotsc, U_m)$ be another $\mathbb{A}_r^m$-contraction on a Hilbert space $\mathcal{K}$ containing $\mathcal{H}$. Then the following are equivalent.
		\begin{enumerate}
			\item $f(\underline{T})x=P_\mathcal{H}f(\underline{U})x$ \ for every $f \in Rat(\overline{\mathbb{A}}_r^m)$ and $ \ x \in \mathcal{H}$ ;
			\item $\underline{T}^kx=P_\mathcal{H}\underline{U}^kx$ \ for every $k \in \mathbb{Z}^m$ and $x \in \mathcal{H}.$
		\end{enumerate}
		
	\end{prop}
	
	\begin{proof}
		$(1) \Rightarrow (2)$ is obvious. We prove that $(2) \Rightarrow (1)$ here. Let $f \in Rat \;(\ov{\mathbb A}_r^m)$ and let $f(z)={\underset{k \in \mathbb{Z}^m}{\sum}}f_kz^k$ as in Proposition \ref{prop:503}. It follows from Lemma $\ref{Functional calculus}$ that 
		\[
		f(\underline{T})=\underset{k \in \mathbb{Z}^m}{\sum}f_k\underline{T}^k \quad \text{ and } \quad  f(\underline{U})=\underset{k \in \mathbb{Z}^m}{\sum}f_k\underline{U}^k.
		\]
		 So, for any $x\in \HS$ we have that 
		\begin{equation*}
			\begin{split}
				f(\underline{T})x
				=\underset{k \in \mathbb{Z}^m}{\sum}f_k\underline{T}^kx
				=\underset{k \in \mathbb{Z}^m}{\sum}f_k(P_\mathcal{H}\underline{U}^kx)
				=P_\mathcal{H}\underset{k \in \mathbb{Z}^m}{\sum}f_k\underline{U}^kx=P_\mathcal{H}f(\underline{U})x .\\
			\end{split}
		\end{equation*}

	\end{proof}

\vspace{0.2cm}
	\section{The Dirichelt problem for the polyannulus $\Arm$}\label{Dirichlet problem}

\vspace{0.3cm}	
	
\noindent Recall that if $X \subset \C^m$ is a compact set, then $C_{\mathbb R}(X)$ denotes the space of real-valued continuous functions on $X$. In this Section, we solve a Dirichlet problem for the polyannulus $\Arm$. Indeed, we show that a function $f$ that is continuous on the distinguished boundary $b\Arm$ can be extended to a function $\hat f$ on $\ov{\mathbb A}_r^m$ such that $\hat f$ is harmonic in each variable separately. We call such a function $\hat f$ a \textit{strongly harmonic function}. We begin with a Dirichlet problem for an annulus $\Ar$.

\begin{thm}[\cite{Axler}, Theorem 10.13]\label{DirichletI}
		Let $f$ be a real-valued continuous function on $\partial \mathbb{A}_r$ given by
		\[ 
		f(z)= \left\{
		\begin{array}{ll}
			g(z) & z \in \mathbb{T} \\
			h(z) & z \in r\mathbb{T}. \\
		\end{array} 
		\right. 
		\]
		There exists a unique $u \in C_\mathbb{R}(\overline{\mathbb{A}}_r)$ such that 
		$ \Delta u=0 \ \mbox{in} \ \mathbb{A}_r \ \mbox{and} \ u=f \ \mbox{on} \ \partial \mathbb{A}_r$. 
		Moreover, the unique solution $u$ is given by 
		\begin{equation} \label{eqn:601}
			u(z)=\underset{\mathbb{T}}{\int}g(w)P_A(z,w)dw+ \underset{\mathbb{T}}{\int}h(rw)P_A(z,rw)dw \,,
		\end{equation}
		where the functions $P_A(.,w)$ and $P_A(.,rw)$ are harmonic on $\mathbb{A}_r$ for each $w \in \mathbb{T}$.
	\end{thm}
\begin{rem} \label{rem:601}
	The above theorem shows that the functions defined on $\ov{\mathbb A}_r$ by
			\begin{equation*}
			u_0(z)=\underset{\mathbb{T}}{\int}P_A(z,w)dw \quad \mbox{and} 	\quad 	u_r(z)=\underset{\mathbb{T}}{\int}P_A(z,rw)dw
		\end{equation*}
	are real-valued continuous and since $\ov{\mathbb A}_r$ is compact, there is $M>0$ such that
\[
|u_0(z)|, |u_r(z)| \leq M \quad \text{ for every } z \in \overline{\mathbb{A}}_r.
\]

\end{rem}
	\vspace{0.2cm}
	 
	In the single-variable setting, the distinguished boundary of $\mathbb{A}_r$ is same as its topological boundary but this is no longer true for the polyannulus $\Arm$. Indeed, the distinguished boundary of $\mathbb{A}_r^m$, when $m\geq 2$, is given by $b\mathbb{A}_r^m=(\partial \mathbb{A}_r)^m$. Note that this is a very thin subset of the topological boundary of $\mathbb{A}_r^m$, yet we shall see that any continuous real-valued function $f$ on $b\Arm$ can have a unique strongly harmonic extension to $\ov{\mathbb A}_r^m$. This is the unique solution to the Dirichlet problem for $\Arm$, which we denote by $\hat f$, and call the \textit{harmonic extension} of $f$. This is the main result of this Section.
	
	\begin{thm}\label{Dirichlet}
	
		Let $m\geq 2$ and let $f$ be a real-valued continuous function on $b\mathbb{A}_r^m$. Then there is a unique real-valued continuous function $u$ on $\overline{\mathbb{A}}_r^m$ such that $\Delta_j u=0 $ in $\mathbb{A}_r^m $ for each $j=1,2,\dotsc, m$ and $ u|_{b\Arm}=f$.
		
		\end{thm}
	
	\begin{proof}
		We prove the statement for $m=2,$ i.e. the biannulus case and the general case follows similarly. Assume that $f$ is a real-valued continuous function on $\partial \mathbb{A}_r \times \partial \mathbb{A}_r$. Then $f$ has the following form:
		\[ 
		f(z,w)= \left\{
		\begin{array}{ll}
			g(z,w) & \mbox{on} \ \mathbb{T}\times \mathbb{T} \\
			h(z,w) &  \mbox{on} \ \mathbb{T} \times r\mathbb{T} \\
			p(z,w) & \mbox{on} \ r\mathbb{T}\times \mathbb{T} \\
			q(z,w) & \mbox{on} \ r\mathbb{T} \times r\mathbb{T}, \\
		\end{array} 
		\right. 
		\]
		where $g,h,p,q$ are continuous functions. Consider the Dirichlet problem 
		\[ 
		\Delta_1 u=0=\Delta_2 u \;\; \mbox{ in } \; \mathbb{A}_r^2 \;\; \mbox{with} \ \ u= \left\{
		\begin{array}{ll}
			g & \mbox{on} \ \mathbb{T}\times \mathbb{T} \\
			0 &  \mbox{on} \ b\mathbb{A}_r^2\setminus \mathbb{T}^2. \\
		\end{array}
		\right.
		\]
		We show that the function defined by 
		\begin{equation*}
			u_{11}(z_1, z_2)=\underset{\mathbb{T}}{\int}\underset{\mathbb{T}}{\int}g(w_1, w_2)P_A(z_1, w_1)P_A(z_2, w_2)dw_1dw_2 \,,
		\end{equation*}
	where $P_A$ is as in (\ref{eqn:601}), solves the above Dirichlet problem. Given $z_1 \in \overline{\mathbb{A}}_r$, define a function $\phi_{z_1}: \mathbb{T} \to \mathbb{R}$ by 
		\begin{equation*}
			 \phi_{z_1}(w)=\underset{\mathbb{T}}{\int}g(w_1, w)P_A(z_1,w_1)dw_1. 
		\end{equation*}
		We show that $\phi_{z_1} \in C(\mathbb{T})$. It follows from Remark \ref{rem:601} that there exists $M>0$ such that for every $z \in \overline{\mathbb{A}}_r$, we have
		\[
\big|\underset{\mathbb{T}}{\int} P_A(z, w)dw \big| \leq  M.
		\]
		Since $g$ is uniformly continuous on $\mathbb{T}^2$, for every $\epsilon >0$ there is a $\delta>0$ such that 
		\begin{equation*}
			|g(w_1, w)-g(w_1, w')| < \frac{\epsilon}{M} \quad \mbox{whenever} \quad |w-w'|<\delta.    
		\end{equation*}
		So, we have
		\begin{equation*}
			\begin{split}
				\phi_{z_1}(w)-\phi_{z_1}(w') & = \underset{\mathbb{T}}{\int} \bigg(g(w_1, w)-g(w_1, w')\bigg) P_A(z_1, w_1)dw_1 \\
								& < \frac{\epsilon}{M}\underset{\mathbb{T}}{\int} P_A(z_1, w_1)dw_1 \\
				& < \epsilon \; \; \mbox{ whenever } \ |w-w'| < \delta. \\
			\end{split}
		\end{equation*}
	Similarly, $\phi_{z_1}(w')-\phi_{z_1}(w)< \epsilon $ whenever  $|w-w'|< \delta$. Thus, we have that 
	\[
	|\phi_{z_1}(w)-\phi_{z_1}(w')|< \epsilon  \quad \mbox{ whenever }  \quad |w-w'|< \delta.
	\]
Consequently, 
		\[
		u_{11}(z_1, z_2)=\underset{\mathbb{T}}{\int} \phi_{z_1}(w_2)P_A(z_2, w_2) dw_2.
		\]  
For a fixed $z_1 \in \overline{\mathbb{A}}_r,$ Theorem \ref{DirichletI} implies that the function $u_{11}(z_1,.)$ is continuous on $\{z_1\} \times \overline{\mathbb{A}_r}$ and $\Delta_2u_{11}(z_1,.)=0$ on $\{z_1\} \times \mathbb{A}_r$ such that 
		\[ 
		u_{11}(z_1,z)= \left\{
		\begin{array}{ll}
			\phi_{z_1}(z) \,, & z \in  \mathbb{T} \\
			0 \,,& z \in  r\mathbb{T}. \\
		\end{array} 
		\right.
		\]
For $(z_1, z_2) \in \mathbb{T}^2$, Theorem \ref{DirichletI} yields that 
\[
g(z_1, z_2)=\underset{\mathbb{T}}{\int}g(w, z_2)P_A(z_1,w)dw
\quad \mbox{and} \quad 
\phi_{z_1}(z_2) = \underset{\mathbb{T}}{\int} \phi_{z_1}(w_2)P_A(z_2, w_2) dw_2 .
\]
Consequently, we have		
		\begin{equation*}
			\begin{split}
				u_{11}(z_1, z_2)=\underset{\mathbb{T}}{\int} \phi_{z_1}(w_2)P_A(z_2, w_2) dw_2 
				 =\phi_{z_1}(z_2) 
				=\underset{\mathbb{T}}{\int}g(w_1, z_2)P_A(z_1,w_1)dw_1
				=g(z_1,z_2).\\
			\end{split}
		\end{equation*}
		Next, we show that for a fixed $z_2 \in \overline{\mathbb{A}}_r, \ \Delta_1u_{11}=0.$ For that, we first write $u_{11}$ as
		\begin{equation*}
			u_{11}(z_1, z_2)=\underset{\mathbb{T}}{\int}\phi_{z_2}(w_1)P_A(z_1, w_1)dw_1 \,,
		\end{equation*}
		where $\phi_{z_2}: \mathbb{T} \to \mathbb{R}$ is defined by
		\begin{equation*}
			\phi_{z_2}(w)=\underset{\mathbb{T}}{\int}g(w, w_2)P_A(z_2,w_2)dw_2. 
		\end{equation*}
By an argument similar to that above in the proof of continuity of $\phi_{z_1}$, we can say that $\phi_{z_2} \in C(\mathbb{T})$. Let us rewrite
	\[
	u_{11}(z_1, z_2)=\underset{\mathbb{T}}{\int} \phi_{z_2}(w_1)P_A(z_1, w_1) dw_1\,.
	\]
	 For a fixed $z_2 \in \overline{\mathbb{A}}_r,$ the one-variable theory (i.e. Theorem \ref{DirichletI}) implies that $\Delta_1u_{11}=0.$ Therefore, we have that 
		\[ 
		\Delta_1 u_{11}=0=\Delta_2 u_{11} \ \mbox{ on } \ \mathbb{A}_r \times \mathbb{A}_r \ \mbox{ with } \ \ u_{11}= \left\{
		\begin{array}{ll}
			g & \mbox{on} \ \mathbb{T}\times \mathbb{T} \\
			0 &  \mbox{on} \ b\mathbb{A}_r^2\setminus \mathbb{T}^2 .\\
		\end{array}
		\right.
		\]
It only remains to show that the function $u_{11}$ is continuous on $\overline{\mathbb{A}}_r^2$. We have already proved that $u_{11}(z_1,.)$ and $u_{11}(.,z_2)$ are continuous functions on $\overline{\mathbb{A}}_r$ for every $z_1, z_2 \in \overline{\mathbb{A}}_r$. Let $\epsilon > 0 $ be given. 
For every $(z_1, z_2), (\alpha_1, \alpha_2) \in \overline{\mathbb{A}}_r^2$, we have 
\begin{equation} \label{eqn:602}
	\begin{split}
		u_{11}(z_1, z_2)-u_{11}(\alpha_1, \alpha_2) &=\underset{\mathbb{T}}{\int}\underset{\mathbb{T}}{\int}g(w_1, w_2)\bigg[P_A(z_1, w_1)P_A(z_2, w_2)-P_A(\alpha_1, w_1)P_A(\alpha_2, w_2)\bigg]dw_1dw_2  \\
&=\underset{\mathbb{T}}{\int}\underset{\mathbb{T}}{\int}g(w_1, w_2)\bigg[P_A(z_1, w_1)P_A(z_2, w_2)-P_A(\alpha_1, w_1)P_A(z_2, w_2)\bigg]dw_1dw_2  \\
&+\underset{\mathbb{T}}{\int}\underset{\mathbb{T}}{\int}g(w_1, w_2)\bigg[P_A(\alpha_1, w_1)P_A(z_2, w_2)-P_A(\alpha_1, w_1)P_A(\alpha_2, w_2)\bigg]dw_1dw_2 \\
&=A+B. \\
	\end{split}
\end{equation}
Since $g(w_1, w_2)$ is continuous on $\mathbb{T}^2$, there is $K>0$ such that $|g(w_1, w_2)|<K$ for all $(w_1, w_2) \in \mathbb{T}^2$.	Also, the uniform continuity of the functions $u_0$ implies that for $\epsilon >0$, there is a $\delta>0$ such that 
\[
|u_0(z)-u_0(w)|<\frac{\epsilon}{2KM}\,, \quad \mbox{whenever } |z-w|< \delta, \text{ for every } z,w \in \overline{\mathbb{A}}_r \,.
\]
From (\ref{eqn:602}) we have that 
\begin{equation*}
	\begin{split}
A&=	\underset{\mathbb{T}}{\int}\underset{\mathbb{T}}{\int}g(w_1, w_2)\bigg[P_A(z_1, w_1)P_A(z_2, w_2)-P_A(\alpha_1, w_1)P_A(z_2, w_2)\bigg]dw_1dw_2\\	
&=	\underset{\mathbb{T}}{\int}\underset{\mathbb{T}}{\int}g(w_1, w_2)P_A(z_2, w_2)\bigg[P_A(z_1, w_1)-P_A(\alpha_1, w_1)\bigg]dw_1dw_2\\
&\leq K\underset{\mathbb{T}}{\int}\underset{\mathbb{T}}{\int}P_A(z_2, w_2)\bigg[P_A(z_1, w_1)-P_A(\alpha_1, w_1)\bigg]dw_1dw_2\\
&= K\underset{\mathbb{T}}{\int}P_A(z_2, w_2)\bigg[\underset{\mathbb{T}}{\int}(P_A(z_1, w_1)-P_A(\alpha_1, w_1))dw_1\bigg]dw_2\\
&=K\underset{\mathbb{T}}{\int}P_A(z_2, w_2)[u_0(z_1)-u_0(\alpha_1)]dw_2\\
&<K\frac{\epsilon}{2KM}\underset{\mathbb{T}}{\int}P_A(z_2, w_2)dw_2 \\
& <\frac{\epsilon}{2}  \quad \mbox{ whenever $|z_1-\alpha_1|<\delta$}.\\	
\end{split}
\end{equation*}
Similarly, we can prove that $B < \epsilon \slash 2 \;$ whenever $|z_2-\alpha_2|<\delta.$ Consequently, we have 
\[
u_{11}(z_1, z_2)-u_{11}(\alpha_1, \alpha_2) < \epsilon \quad \mbox{whenever $|z_1-\alpha_1|<\delta, |z_2-\alpha_2|<\delta$}.
\]
Interchanging the roles of $z_i$ and $\alpha_i$ for $i=1,2,$ we get 
\[
|u_{11}(z_1, z_2)-u_{11}(\alpha_1, \alpha_2)| < \epsilon \quad \mbox{whenever $|z_1-\alpha_1|<\delta, |z_2-\alpha_2|<\delta$}.
\]
Since $\epsilon>0$ is arbitrary, we have that $u_{11}$ is a continuous function on $\overline{\mathbb{A}}_r^2.$
		Using the continuity of $u_0$ and $u_r$ on $\overline{\mathbb{A}}_r,$ we can show in a similar way that the following Dirichlet type problem has a unique solution in $C_\mathbb{R}(\overline{\mathbb{A}}_r^2)$ which are given below.
		
		\begin{enumerate}
			\item  
			\[ 
			\Delta_1 u=0=\Delta_2 u \;\; \mbox{ on } \ \mathbb{A}_r^2 \; \mbox{ with} \ \ u= \left\{
			\begin{array}{ll}
				h & \mbox{on} \ \mathbb{T}\times r\mathbb{T} \\
				0 &  \mbox{on} \ b\mathbb{A}_r^2\setminus \mathbb{T}\times r\mathbb{T}\,, \\
			\end{array}
			\right.
			\]
			has the solution 
			\begin{equation*}
				u_{1r}(z_1, z_2)=\underset{\mathbb{T}}{\int}\underset{\mathbb{T}}{\int}h(w_1,r w_2)P_A(z_1, w_1)P_A(z_2, rw_2)dw_1dw_2 \,.
			\end{equation*}
			\item 
			\[ 
			\Delta_1 u=0=\Delta_2 u \;\; \mbox{ on} \ \mathbb{A}_r^2 \; \mbox{ with} \ \ u= \left\{
			\begin{array}{ll}
				p & \mbox{on} \ r\mathbb{T}\times \mathbb{T} \\
				0 &  \mbox{on} \ b\mathbb{A}_r^2\setminus r\mathbb{T}
				\times \mathbb{T}\;,\\
			\end{array} 
			\right.
			\]
			has the solution 
			\begin{equation*}
				u_{r1}(z_1, z_2)=\underset{\mathbb{T}}{\int}\underset{\mathbb{T}}{\int}p(rw_1, w_2)P_A(z_1, rw_1)P_A(z_2, w_2)dw_1dw_2 \,.
			\end{equation*}
			
			\item
			\[ 
			\Delta_1 u=0=\Delta_2 u \;\; \mbox{ on} \ \mathbb{A}_r^2 \;\; \mbox{with} \ \ u= \left\{
			\begin{array}{ll}
				q & \mbox{ on } \; r\mathbb{T}\times r\mathbb{T} \\
				0 &  \mbox{ on} \; b\mathbb{A}_r^2\setminus r\mathbb{T}\times r \mathbb{T} \,, \\
			\end{array} 
			\right.
			\]
			has the solution
			\begin{equation*}
				u_{rr}(z_1, z_2)=\underset{\mathbb{T}}{\int}\underset{\mathbb{T}}{\int}q(rw_1, rw_2)P_A(z_1, rw_1)P_A(z_2, rw_2)dw_1dw_2 \,.
			\end{equation*}
		\end{enumerate}
		Putting everything together, we have that $u=u_{11}+u_{1r}+u_{r1}+u_{rr} \in C_\mathbb{R}(\overline{\mathbb{A}}_r^2)$ is a solution to the following Dirichlet problem: 
		\[ 
		\Delta_1 u=0=\Delta_2 u \;\; \mbox{ on} \ \mathbb{A}_r \times \mathbb{A}_r \; \mbox{ with} \ \ u= f= \left\{
		\begin{array}{ll}
			g & \mbox{on} \ \mathbb{T}\times \mathbb{T} \\
			h &  \mbox{on} \ \mathbb{T} \times r\mathbb{T} \\
			p & \mbox{on} \ r\mathbb{T}\times \mathbb{T} \\
			q & \mbox{on} \ r\mathbb{T} \times r\mathbb{T}\,. \\
		\end{array} 
		\right. 
		\]
		The uniqueness of this solution follows from the uniqueness of the individual solutions. For the general polyannulus $\mathbb{A}_r^m$ case, one must have $2^m$ individual Dirichlet type problems and adding the solutions of these $2^m$ Dirichlet problems, we have the unique solution of original Dirichlet problem. The proof is now complete.
		
	\end{proof}
\noindent	Next, we prove some properties of the harmonic extension.
	
	\begin{cor}\label{Max-mod} Given $u \in C_\mathbb{R}(b\mathbb{A}_r^m),$ we have that $ \|\hat{u}\|_{\infty, \, \overline{\mathbb{A}}_r^m}=\|u\|_{\infty,\, b\mathbb{A}_r^m}$.
	\end{cor}
	
	\begin{proof}

		Given $u \in C_\mathbb{R}(b\mathbb{A}_r^m),$ we have $\Hat{u} \in C(\overline{\mathbb{A}}_r^m).$ Let $(a_1, \dotsc, a_m) \in \overline{\mathbb{A}}_r^m$ be such that 
		\[ 
		|\hat{u}(a_1, \dotsc, a_m)|=\sup\{|\hat{u}(z_1, \dotsc, z_m)| : \ (z_1, \dotsc, z_m) \in \overline{\mathbb{A}}_r^m   \}.
		 \]
		Define a function $\hat{u}_1(z)=\hat{u}(z, a_2, \dotsc, a_m)$ which is real-valued and continuous on $\overline{\mathbb{A}}_r$. Following the proof of Theorem \ref{Dirichlet}, we have that for fixed $a_2, \dotsc, a_m$ in $\overline{\mathbb{A}}_r$, $\Delta_1\hat{u}=0$ and hence we have that $\Delta\hat{u}_1=0$ on $\mathbb{A}_r$. Since a harmonic function satisfies the Maximum-modulus principle, we have that the maximum modulus of $\hat{u}_1$, which is attained at $a_1$, must belong to $\partial\mathbb{A}_r$. Similarly, one can show that $a_2, \dotsc, a_m $ belong to $\partial \mathbb{A}_r$. Consequently, $(a_1, \dotsc, a_m) \in b\mathbb{A}_r^m$. Since $\hat{u}=u$ on $b \mathbb{A}_r^m,$ we have that 
		\[
		\hat{u}(a_1, a_2, \dotsc, a_m)|=\sup\{|\hat{u}(z_1, z_2, \dotsc, z_m)| : \ (z_1, z_2, \dotsc, z_m) \in \overline{\mathbb{A}}_r^m\}=|u(a_1, a_2, \dotsc, a_m)|.  
		\]
		Thus, combining everything we have that $
			\|\hat{u}\|_{\infty, \; \overline{\mathbb{A}}_r^m}=\|u\|_{\infty,\; b\mathbb{A}_r^m}$ and the proof is complete.
		
	\end{proof}

	\begin{cor}\label{Linear}
	
		Given $u_1, u_2 \in C_\mathbb{R}((\partial \mathbb{A}_r)^m)$ and $\alpha_1, \alpha_1 \in \mathbb{R},$ the unique harmonic extension of $\alpha_1u_1+\alpha_2u_2$ is $\alpha_1\hat{u}_1+\alpha_2\hat{u}_2\,$.
				
	\end{cor}
	
	\begin{proof}
	
		The proof is straightforward as the function $\alpha_1\hat{u}_1+\alpha_2\hat{u}_2$ solves the Dirichlet problem 
		\[ 
		\Delta_j u=0 \;\; \mbox{ on } \; \mathbb{A}_r^m \ (1\leq j \leq m) \; \mbox{ and } \; u=\alpha_1u_1+\alpha_2u_2 \ \mbox{ on } \  b\mathbb{A}_r^m.
		 \]
		The rest follows from the uniqueness of the solution. 
	\end{proof}
	Note that given $u \in C_\mathbb{R}(b\mathbb{A}_r^m),$ we have a unique strongly harmonic extension $\hat{u}$ on $\overline{\mathbb{A}}_r^m$. We can  form a bounded linear functional on $C_\mathbb{R}(b\mathbb{A}_r^m)$ by integrating $\hat{u}$ with respect to a regular complex Borel measure on $\overline{\mathbb{A}}_r^m$ to obtain a  regular Borel measure on  $b\mathbb{A}_r^m$. This suggests a way to construct a measure supported at b$\mathbb{A}_r^m$ from an existing measure supported at $\overline{\mathbb{A}}_r^m.$ The following proposition explains this.
	
	
	\begin{prop}\label{measure}
		Let $\nu$ be a regular complex Borel measure on $\overline{\mathbb{A}}_r^m.$ Then there is a unique regular complex Borel measure $\hat{\nu}$ on $b\mathbb{A}_r^m$ such that 
		
		\[
		\underset{\overline{\mathbb{A}}_r^m}{\int}\hat{u} \ d\nu=\underset{b\mathbb{A}_r^m}{\int}u \ d\hat{\nu} \ \ \ \  \ \ \text{ for all } \; u \in C_\mathbb{R}(b\mathbb{A}_r^m).
		\]
	\end{prop}

	\begin{proof}
		The map $\Lambda :C_\mathbb{R}(b\mathbb{A}_r^m)\to \mathbb{C}$ given by
		\begin{equation*}
			\Lambda(u):=\underset{\overline{\mathbb{A}}_r^m}{\int}\hat{u} \ d\nu
		\end{equation*}
		is a bounded linear functional. Indeed, the linearity of $\Lambda$ follows from Corollary \ref{Linear} and the boundedness of $\Lambda$ is obtained by an application of Corollary \ref{Max-mod}. So, we have that
		\begin{equation*}
			\begin{split}
				|\Lambda(u)|
				& \leq \|\hat{u}\|_{\infty, \overline{\mathbb{A}}_r^m}|\nu(\overline{\mathbb{A}}_r^m)| \\
		& \leq \|u\|_{\infty, b\mathbb{A}_r^m}\|\nu\|.
			\end{split}
		\end{equation*}
		Consequently, we have that $\|\Lambda\| \leq \|\nu\|.$ The desired conclusion follows from the Riesz-Kakutani representation theorem (see Theorem \ref{thm:2RK}).
		
	\end{proof}

	\section{Dilation of normal $\Ar$-contractions and doubly commuting $2\times 2$ scalar $\Ar$-contractions }\label{Normal}
	
	\vspace{0.3cm}
	\noindent In this Section, we present the main results of this paper. We show that any commuting normal $\Ar$-contractions $T_1, \dots , T_m$ simultaneously dilate to commuting $\Ar$-unitaries $U_1, \dots , U_m$. Lemma \ref{Polyannulus-each annulus} tells us that $(U_1, \dots , U_m)$ is an $\Arm$-unitary if and only if each $U_j$ is an $\Ar$-unitary. Therefore, it follows that $f(T_1, \dots , T_m)$ is a compression of $f(U_1, \dots , U_m)$ for every $f$ in $Rat \, (\ov{\mathbb A}_r^m)$ and consequently $\ov{\mathbb A}_r^m$ is a complete spectral set for $(T_1, \dots , T_m)$ by Arveson's theorem (Theorem \ref{thm:W.Arveson}). Therefore, $\ov{\mathbb A}_r^m$ is a spectral set for $(T_1, \dots , T_m)$ and naturally von Neumann's inequality holds on $\ov{\mathbb A}_r^m$, i.e.
	\[
	\|f(T_1, \dots , T_n)\| \leq \sup_{z\in \ov{\mathbb A}_r^m} \; |f(z)| = \|f\|_{\infty, \ov{\mathbb A}_r^m} \, ,
	\]
	for every $f\in Rat\,(\ov{\mathbb A}_r^m)$. We further show that such a dilation is possible when $(T_1, \dots , T_m)$ is a subnormal $\Arm$-contraction or when $T_1, \dots , T_n$ are doubly commuting $2 \times 2$ scalar $\Ar$-contractions. We start with a lemma whose proof follows directly from the properties of Taylor joint spectrum.
	\begin{lem}
		An $m$-tuple $\underline{N}=(N_1, \dotsc, N_m)$ of commuting normal operators is an $\mathbb{A}_r^m$-contraction if and only if each $N_j$ is a normal $\mathbb{A}_r$-contraction.
	\end{lem}

	\begin{thm}\label{main}
		Every normal $\mathbb{A}_r^m$-contraction dilates to an $\mathbb A_r^m$-unitary.
	\end{thm}
	\begin{proof}
		Let $\underline{N}=(N_1, \dotsc, N_m)$ be a normal $\mathbb A_r^m$-contraction acting on a Hilbert space $\mathcal{H}.$ Then there exists a unique spectral measure $E$ relative to $(\overline{\mathbb{A}}_r^m, \mathcal{H})$ such that 
		\begin{equation*}
			N_j=\underset{\overline{\mathbb{A}}_r^m}{\int}z_jdE(z) \ \ \text{for all} \ \ j=1,2,\dotsc, m,
		\end{equation*}
		where $z_j$ denotes the projection onto the $j$-th coordinate. For $x, y \in \mathcal{H}$, define a regular complex Borel measure $\nu_{x,y}$ on $\overline{\mathbb{A}}_r^m$ by $\nu_{x,y}(S)=\langle E(S)x,y\rangle.$ It follows from Proposition \ref{measure} that there is a unique regular complex Borel measure $\hat{\nu}_{x,y}$ on $b\mathbb{A}_r^m$ such that 
		\[
		\underset{\overline{\mathbb{A}}_r^m}{\int}\hat{u} \ d\nu_{x, y}=\underset{b\mathbb{A}_r^m}{\int}u \ d\hat{\nu}_{x,y} \ \ \bigg(u \in C_\mathbb{R}(b\mathbb{A}_r^m)\bigg).  
		\]
		If $\Delta$ is a Borel set in $b\mathbb{A}_r^m$, then the map 
		$
		B_\Delta(x,y)=\hat{\nu}_{x,y}(\Delta)
		$
		defines a bounded sesquilinear form on $\mathcal{H}.$ By Reisz representation theorem, we have that there is a unique operator $A(\Delta) \in B(\mathcal{H})$ such that $\hat{\nu}_{x,y}(\Delta)=\langle A(\Delta)x, y\rangle $ for every $x, y \in \mathcal{H}.$ Again, it is evident that the map 
		\[
		A: \Delta \mapsto A(\Delta), \ \Delta \overset{Borel}{\subseteq} b\mathbb{A}_r^m,
		\]
		is a regular positive operator valued measure. Then by Theorem \ref{Naimark}, there is a space $\mathcal{K},$ an isometry $V:\mathcal{H} \to \mathcal{K}$ and a regular self-adjoint spectral measure $F$ on $b\mathbb{A}_r^m$ such that 
		\[
		A(\Delta) =V^*F(\Delta)F
		\]
		for all Borel sets $\Delta \subseteq b\mathbb{A}_r^m.$ Consequently $\hat{\nu}_{x,y}=F_{Vx, Vy}\,,$ because
		\[
		\hat{\nu}_{x,y}(\Delta)=\langle A(\Delta)x, y\rangle =\langle V^*F(\Delta)Vx, y\rangle=\langle F(\Delta)Vx, Vy\rangle =F_{Vx, Vy}(\Delta)
		\]
		for all Borel sets $\Delta \subseteq b\mathbb{A}_r^m.$
		Now for each $j=1, \dots , m$, set $U_j: \mathcal{K} \to \mathcal{K}$ to be the operator
		\[
		\langle U_jx,y \rangle := \underset{b\mathbb{A}_r^m}{\int}z_j\,dF_{x,y} \,.
		\]
		It follows from the spectral theorem that $\underline{U}=(U_1, \dotsc, U_m)$ is a commuting tuple of normal operators with $\sigma_T(U_1, \dots , U_m)\subseteq b\mathbb A_r^m$. Thus, $\underline{U}$ is an $\mathbb A_r^m$-unitary. Now we show that the tuple $\underline{U}$ is indeed our desired normal $b\mathbb A_r^m$-dilation. For $(k_1, \dotsc, k_m) \in \mathbb{Z}^m,$ let us define a function $g$ on $\overline{\mathbb{A}}_r^m$ by $g(z_1, \dotsc, z_m)=z_1^{k_1}\dotsc z_m^{k_m}$ which is evidently analytic. If we write $g=f_1+if_2$, where $f_1,f_2$ are real-valued, then it follows from Theorem \ref{Dirichlet} that $f_1=\hat f_1$ and $f_2=\hat f_2$. Thus, $\hat g =\hat f_1 +i \hat f_2=g$. So, we have that
		\begin{equation*}
			\begin{split}
				\langle (N_1^{k_1}\dotsc N_m^{k_m})x,y\rangle &=
				\langle g(N_1, \dotsc, N_m)x,y\rangle \\  
				&=\underset{\overline{\mathbb{A}}_r^m}{\int}g \ d\nu_{x,y}\\
				&=\underset{b\mathbb{A}_r^m}{\int}\hat{g} \ d\hat{\nu}_{x,y} \quad \quad \; \quad [\text{ by Proposition }\ref{measure}]\\
				&=\underset{b\mathbb{A}_r^m}{\int} g \ d\hat{\nu}_{x,y} \quad \quad \; \quad \mbox{[since $g=\hat{g}$]}\\
				&=\underset{b\mathbb{A}_r^m}{\int} g \ dF_{Vx,Vy}\\
				&=\langle g(\underline{U})Vx, Vy\rangle \quad \quad \mbox{[since $F$ is a spectral measure]}\\ 
				&= \langle (U_1^{k_1} \dotsc U_m^{k_m})Vx,Vy\rangle.\\
			\end{split}
		\end{equation*}
		Hence we have that $\underline{N}^k=V^*\underline{U}^kV$ for every $k\in \mathbb{Z}^m.$ So, it follows from Proposition \ref{Power} that $$f(\underline{N})x=V^*f(\underline{U})Vx$$ for every $x\in \HS$ and for every rational function $f=p\slash q$ such that $q$ does not have any zero in $\overline{\mathbb{A}}_r^m$. The proof is now complete.
	\end{proof}

	\noindent The preceding dilation theorem holds even for a wider class of $\Arm$-contractions. Indeed, below we show that a subnormal $\Arm$-contraction can have a normal $b\Arm$-dilation. Recall that a commuting $m$-tuple of operators $\underline{S}=(S_1, \dotsc, S_m)$ acting on a Hilbert space $\mathcal{H}$ is said to be \textit{subnormal} if there exist a Hilbert space $\mathcal{K} \supseteq \mathcal{H}$ and commuting normal operators  $N_1, \dotsc, N_m$ in $B(\mathcal{K})$ such that $\HS$ is a joint invariant subspace of $(N_1, \dots , N_m)$ and that $N_j|\mathcal{H}=S_j$ for $1 \leq j \leq m$. Among all such commuting normal extensions of the commuting tuple $\underline{S}$, there is a minimal one, which is unique upto unitary equivalence. Thus, an $\Arm$-contraction $\underline{S}=(S_1, \dotsc, S_m)$ is said to be a \textit{subnormal $\mathbb{A}_r^m$-contraction} if the $m$-tuple $\underline{S}$ is subnormal.
	
	\begin{thm}\label{subnormal}
		Every subnormal $\mathbb{A}_r^m$-contraction admits a normal $b\mathbb A_r^m$-dilation.
	\end{thm}
	
	\begin{proof}
		Let $\underline{S}=(S_1, \dotsc, S_m)$ be a subnormal $\mathbb{A}_r^m$-contraction on a Hilbert space $\HS$ and let $\underline{N}=(N_1, \dotsc, N_m)$ be its minimal normal extension on the minimal space $\mathcal{K} \supseteq \mathcal{H}$. The spectral inclusion theorem for commuting tuple of subnormal operators \cite{Putinar} yields that
	\[
		\sigma_T(\underline{N}) \subseteq \sigma_T(\underline{S}) \subseteq \overline{\mathbb{A}}_r^m.
	\]
		This shows that the $m$-tuple $\underline{N}$ is an $\mathbb{A}_r^m$-contraction. Since each $S_j$ is an invertible operator by being an $\Ar$-contraction and $N_j|\mathcal{H}=S_j$, it follows that the space $\mathcal{H}$ is invariant under $N_j^{-1}$ too. Hence, $S_j^{-1}=N_j^{-1}|\mathcal{H}$ for $1 \leq j \leq m$. Consequently, for any $k=(k_1, \dotsc, k_m) \in \mathbb{Z}^m,$ we have
		\[
		S_1^{k_1}\dotsc S_m^{k_m}=N_1^{k_1}\dotsc N_m^{k_m}|\mathcal{H}.
		\]
		Since $(N_1, \dots , N_m)$ dilates to an $\mathbb A_r^m$-unitary by Theorem \ref{main}, the desired conclusion follows. 
	\end{proof}
	
	It is merely said that if $S_1, \dots , S_m$ are commuting subnormal operators acting on a Hilbert space $\HS$, then the tuple $(S_1, \dots , S_m)$ may not be a subnormal tuple. Thus, if $S_1, \dots , S_m$ are commuting subnormal $\Ar$-contractions, then $(S_1, \dots , S_m)$ may not be a subnormal $\Arm$-contraction. Naturally, $S_1, \dots , S_m$ may not dilate simultaneously to commuting $\Ar$-unitaries $U_1, \dots , U_m$ acting on a bigger Hilbert space. This is because, if they do then $(S_1, \dots ,S_m)$ dilates to the $\Arm$-unitary $(U_1, \dots , U_m)$ and once again by Arveson's theorem (Theorem \ref{thm:W.Arveson}), $\ov{\mathbb A}_r^m$ is a complete spectral set for $(S_1, \dots ,S_m)$, which further implies that $\ov{\mathbb A}_r^m$ is a spectral set for $(S_1, \dots , S_m)$, a contradiction. Interestingly, such a dilation exists if we consider doubly commuting subnormal $\Ar$-contractions $S_1, \dots , S_m$.
	\begin{thm} \label{subnormal1}
		Let, $S_1, \dotsc, S_m$ be doubly commuting subnormal $\mathbb{A}_r$-contractions on $\mathcal{H}$. Then the $m$-tupe $\underline{S}=(S_1, \dotsc, S_m)$ possesses a normal $b\Arm$-dilation. 
	\end{thm}
	\begin{proof}
	
		We have from the literature (e.g. \cite{Conway}) that any doubly commuting $m$-tuple of contractions $\underline{S}=(S_1, \dotsc, S_m)$ is a subnormal tuple. Let $\underline{N}=(N_1, \dotsc, N_m)$ on $\mathcal{K} \supseteq \mathcal{H}$ be the minimal normal extension of $(S_1, \dots , S_m)$. The spectral inclusion theorem for commuting subnormal operators \cite{Putinar} and properties of Taylor-joint spectrum imply that
		\[
		\sigma_T(\underline{N}) \subseteq \sigma_T(\underline{S}) \subseteq \sigma(S_1) \times \dotsc \times \sigma(S_m) \subseteq \overline{\mathbb{A}}_r^m.
		\]
		Thus, $\underline{N}$ is an $\mathbb{A}_r^m$-contraction. By an argument similar to that given in the proof of Theorem \ref{subnormal}, we have for any $k=(k_1, \dotsc, k_m) \in \mathbb{Z}^m$,
		\[
		S_1^{k_1} \dotsc S_m^{k_m}=N_1^{k_1} \dotsc N_m^{k_m}|\mathcal{H}.
		\]
		The rest of the proof follows from Theorem \ref{main}.

	\end{proof}
	
	\subsection{Dilation of doubly commuting $2 \times 2$ scalar matrices associated with an annulus}\label{2*2}
	We have found simultaneous normal $b\Ar$-dilation for commuting normal $\Ar$-contractions. A natural question arises, what happens if we consider doubly commuting $\Ar$-contractions $T_1, \dots , T_m$ instead. Needless to mention that a commuting tuple of normal operators is always doubly commuting. Thus, our question is seeking a generalization if we can find simultaneous $\Ar$-unitary dilation for doubly commuting $\Ar$-contractions. In this Subsection, we settle this issue at least for a doubly commuting $m$-tuple of $2 \times 2$ scalar $\mathbb{A}_r$-contractions. Evidently then $\ov{\mathbb A}_r^m$ becomes a complete spectral set for such an $m$-tuple and hence a spectral set too. Before going to that, we show a canonical way of constructing a commuting family of $2 \times 2$ scalar $\Ar$-contractions that are all upper triangular matrices. The following result from \cite{Misra} plays the central role in the construction. However, we are able to find dilation for $2 \times 2$ doubly commuting $\Ar$-contractions only.
	
	\begin{thm}[\cite{Misra}, Corollary $1.1'$]\label{Misra_cor} For any $w \in \mathbb{A}_r \,,$ the set $\overline{\mathbb{A}}_r$ is a spectral set for $\begin{bmatrix}
			w & h(w) \\
			0 & w 
		\end{bmatrix} $ if and only if  
		$
		|h(w)| \leq \hat{K}_r(\overline{w},w)^{-1}
		$, where $h:{\Ar} \rightarrow \C$ is any function.
	\end{thm}
		
\begin{eg} \label{Example}

	Let $L^2(\partial \mathbb{A}_r)$ be the Hilbert space of Lebesgue measurable functions on $\partial \mathbb{A}_r$ that are square integrable where the inner product is defined by 
	\begin{equation*}
		\langle f,g\rangle=\frac{1}{2\pi}\overset{2\pi}{\underset{0}{\int}}f(e^{it})\overline{g(e^{it})}dt + \frac{1}{2\pi}\overset{2\pi}{\underset{0}{\int}}f(re^{it})\overline{g(re^{it})}dt\,.
	\end{equation*}
	Let $H^2(\partial \mathbb{A}_r)$ be the Hardy space consisting of holomorphic functions from $L^2(\partial \Ar)$. Then	the space $H^2(\partial \mathbb{A}_r)$ is a reproducing kernel Hilbert space (see \cite{Misra}) with the kernel function defined by
	 
	\begin{equation*}
		\hat{K}_r(z,w)=\overset{\infty}{\underset{n=-\infty}{\sum}}\frac{(z\overline{w})^n}{1+r^{2n}}, \quad z,w \in \Ar \,.
	\end{equation*}

For every $w \in \mathbb{A}_r,$ we have
\begin{equation}\label{eqn 701}
\begin{split}
\hat{K}_r(\overline{w},w)
= \overset{\infty}{\underset{n=-\infty}{\sum}}\frac{|w|^{2n}}{1+r^{2n}}
=\overset{\infty}{\underset{n=0}{\sum}}\frac{|w|^{2n}}{1+r^{2n}} + \overset{\infty}{\underset{n=1}{\sum}}\frac{|r|^{2n}}{|w|^{2n}(1+r^{2n})}
& \leq \overset{\infty}{\underset{n=0}{\sum}}|w|^{2n} + \overset{\infty}{\underset{n=1}{\sum}}|r\slash w|^{2n} \\
&=\frac{1}{1-|w|^2}+\frac{|w|^2}{|w|^2-r^2}\\
& \leq \frac{1}{1-|w|^2}+\frac{1}{|w|^2-r^2}\\
&= \frac{1-r^2}{(1-|w|^2)(|w|^2-r^2)}\\
& \leq \frac{1}{(1-|w|^2)(|w|^2-r^2)}.
\end{split}
\end{equation}
Now let us consider the following $2 \times 2$ commuting scalar matrices:
\[
T_w=\begin{pmatrix}
w &  (1-|w|^2)(|w|^2-r^2)\\
0 & w
\end{pmatrix}, \;\; \quad w \in \overline{\mathbb{A}}_r.
\]
First we prove that each $T_w$ is an $\mathbb{A}_r$-contraction. If $w \in \partial \mathbb{A}_r,$ the matrix $T_w$ is normal with $\sigma(T_w)=\{w\} \subset \partial \Ar$. Therefore, $T_w$ is an $\mathbb{A}_r$-contraction, in particular an $\Ar$-unitary. Again, suppose $w \in \mathbb{A}_r$. By Theorem \ref{Misra_cor}, we have that $\overline{\mathbb{A}}_r$ is a spectral set for $T_w$ if and only if $(1-|w|^2)(|w|^2-r^2) \leq \hat{K}_r(\overline{w},w)^{-1}$, which is true by Equation (\ref{eqn 701}). Thus, each $T_w$ is an $\mathbb{A}_r$-contraction.
	
	\end{eg}

	\begin{thm} \label{DC_2}
	
	 Any $m$-tuple of doubly commuting $2 \times 2$ scalar $\mathbb{A}_r$-contractions has a simultaneous $\mathbb{A}_r$-unitary dilation.
	  
	\end{thm}
	\begin{proof}
	
		Let $B_1, \dotsc, B_m$ be doubly commuting $2 \times 2$ scalar $ \mathbb{A}_r$-contractions. Since they commute, they can have simultaneously upper triangular form with respect to some fixed orthonormal basis, i.e. without loss of generality
		\[
		B_j = \begin{pmatrix}
			a_j & c_j \\ 0 & b_j\\
		\end{pmatrix}, \ \ \ (j=1, 2, \dotsc, m).
		\]
		If each $c_j=0$ then $(B_1, \dotsc, B_m)$ is a normal $\mathbb{A}_r^m$-contraction and the desired conclusion follows from Theorem \ref{main}. So, we assume at least one of them say $c_1 \ne 0$. For all $j \ne 1,$ we have that $B_jB_1^*=B_1^*B_j$ which implies that
		\begin{equation*}
			\begin{split}
				 & \begin{pmatrix}
					a_j & c_j \\ 0 & b_j\\
				\end{pmatrix}\begin{pmatrix}
					\overline{a_1} & 0 \\ \overline{c_1} & \overline{b_1}\\
				\end{pmatrix} =\begin{pmatrix}
					\overline{a_1} & 0 \\ \overline{c_1} & \overline{b_1}\\
				\end{pmatrix}\begin{pmatrix}
					a_j & c_j \\ 0 & b_j\\
				\end{pmatrix} \\
				 \text{i.e. } & \begin{pmatrix}
					a_j\overline{a_1}+c_j\overline{c_1} & c_j\overline{b_1} \\ b_j\overline{c_1} & b_j\overline{b_1}\\
				\end{pmatrix}=\begin{pmatrix}
					\overline{a_1}a_j & \overline{a_1}c_j \\ \overline{c_1}a_j & \overline{b_1}b_j\\
				\end{pmatrix}.\\
			\end{split}
		\end{equation*}
		Consequently, $c_j=0$ and $a_j=b_j$ for every $j \neq 1$. Therefore, 
		\[
		B_1 = \begin{pmatrix}
			a_1 & c_1 \\ 0 & b_1\\
		\end{pmatrix} \; \; \ \mbox{ and } \;\; \ B_j = a_j\begin{pmatrix}
			1 & 0 \\ 0 & 1\\
		\end{pmatrix}, \ \  \ \ (j\neq 1).
		\]
		Then, Agler's dilation theorem \cite{Agler} implies that there exist a Hilbert space $\mathcal{K} \supseteq \C^2$ and an $\mathbb{A}_r$-unitary $N_1$ on $\mathcal{K}$ such that
		\[
		f(B_1)=P_\mathcal{H}f(N_1)|_{\mathcal{H}}
		\]
		for every rational function $f$ with poles off $\overline{\mathbb{A}}_r.$ Set $N_j=a_jI_\mathcal{K}$ for $j \ne 1.$ Then $\underline{N}=(N_1, \dotsc, N_m)$ on $\mathcal{K}$ is a normal $\mathbb{A}_r^m$-contraction which dilates to an $\Arm$-unitary by Theorem \ref{main}. Hence, the proof is complete.
		
	\end{proof}
	

For a compact set $X\subset \C^m$ and a commuting tuple of operators $(T_1, \dots , T_m)$ for which $X$ is a spectral set, the aim of rational dilation is to realize $f(T_1, \dots , T_m)$ as a compression of $f(N_1,\dots , N_m)$ for every $f\in Rat\,(X)$, where $(N_1,\dots , N_m)$ is a commuting normal tuple with $\sigma_T(N_1, \dots , N_m)\subseteq bX$. Now if $X=\ov{\mathbb A}_r^m$ and if $B_1, \dots , B_m$ are doubly commuting $2 \times 2$ scalar matrices such that each $B_j$ is an $\Ar$-contraction, then we show that for every finite dimensional subspace $\mathcal R$ of $Rat\,(\Arm)$ there is an $\Arm$-unitary $(U_1, \dots , U_m)$ consisting of scalar matrices such that $f(B_1, \dots , B_m)$ is a compression of $f(U_1, \dots , U_m)$ for every $f\in \mathcal R$. Note that the tuple $(B_1, \dots , B_m)$ is an $\Arm$-contraction by Theorem \ref{DC_2}, because, it dilates to an $\Arm$-unitary. In this context, we state an interesting and useful result due to Hartz and Lupini.

\begin{thm}(\cite{Hartz}, Theorem 3.3)\label{Hartz1}
        Let $A$ be a commutative unital $C^*$-algebra and $S \subseteq A$ be an
operator system with dim$(S) < \infty$ and let $\phi: S \to
\mathcal{B}(\mathcal{H})$ be a unital completely positive map with
dim$(\mathcal{H})< \infty.$ Then $\phi$ dilates to a finite-dimensional
representation of $A,$ i.e. there exist a finite dimensional space
$\mathcal{K}$ containing $\mathcal{H}$ and a unital $*$-representation
$\pi: A \to \mathcal{B}(\mathcal{K})$ such that
$\phi(s)=P_\mathcal{H}\pi(s)|_{\mathcal{H}}$ for all $s \in S.$
\end{thm}

Now we state and prove our desired result and conclude this Section.
        \begin{thm}
Let $B_1, \dotsc , B_m$ be doubly commuting $2 \times 2$ scalar $\mathbb{A}_r$-contractions acting on $\mathcal{H}=\C^2$. Let $\mathcal{R} \subset Rat(\overline{\mathbb{A}}_r^m)$ be a finite
dimensional subspace. Then there exist a finite dimensional Hilbert
space $\mathcal{K} \supseteq \mathcal{H}$ and an $m$-tuple $(U_1, \dotsc, U_m)$ of commuting $\mathbb{A}_r$-unitaries such that
                \[
                f(B_1, \dotsc, B_m)=P_\mathcal{H}f(U_1, \dotsc, U_m)|\mathcal{H}
                \]
                for every $f \in \mathcal{R}.$
        \end{thm}
\begin{proof}
We apply Theorem \ref{Hartz1} with $A=C(b\mathbb{A}_r^m)$ and
\[
S=\text{span}\{1, f, \overline{f} : f \in \mathcal{R}\} \subset
C(b\mathbb{A}_r^m).
\]
By Theorem \ref{DC_2}, there exist a Hilbert space $\mathcal{K}_0
\supseteq \mathcal{H}$ and commuting $\Ar$-unitaries
$N_1, \dotsc, N_m $ acting on $\mathcal{K}_0$ such that
\[
f(B_1, \dotsc, B_m)=P_{\mathcal{H}}f(N_1, \dotsc, N_m)|_{\mathcal{H}}
\]
for all $f \in Rat(\mathbb{A}_r^m)$. Now consider the map
$
\phi_N:C(b\mathbb{A}_r^m) \to \mathcal{B}(\mathcal{K})$ defined by
\[
\phi_N(g)=g(N_1, \dotsc, N_m),
\]
which is continuous and is a representation of $C(b\Arm)$. Then the map $\phi: S \to
\mathcal{B}(\mathcal{H})$, defined by
$
\phi(q)=P_\mathcal{H}\phi_N(q)|_{\mathcal{H}}$, is a unital completely positive map with dim$(\mathcal{H}) < \infty.$ By
Theorem \ref{Hartz1}, $\phi$ dilates to a finite-dimensional
representation $\pi$ of $C(b\mathbb{A}_r^m),$ such that
\[
f(B_1, \dotsc, B_m)=\phi(f)=P_\mathcal{H}\pi(f)|_{\mathcal{H}}, \quad \text{for every } f \in \mathcal{R}.
\]
If we define $U_j=\pi(z_j)$ for $j=1,2, \dotsc,
m,$ then $U_1, \dots , U_m$ are commuting matrices and we have $f(U_1, \dotsc, U_m)=\pi(f)$ for all $f$  in
$Rat(\overline{\mathbb{A}}_r^m)$. Also, $\pi$ is a homomorphism and thus $(U_1, \dots , U_m)$ is an $\Arm$-unitary. Hence, each $U_j$ is an $\Ar$-unitary and the proof is complete.

\end{proof}	
	
\vspace{0.3cm}	

\noindent \textbf{Concluding remarks.} A next step to the results of this article should be an attempt to determine if a doubly commuting $\Arm$-contraction possesses an $\Arm$-unitary dilation. Note that if $A_1, \dots , A_m$ are doubly commuting $\Ar$-contractions, then it does not follow that the tuple $(A_1, \dots , A_m)$ is an $\Arm$-contraction. However, in this paper we have made some progress by obtaining such a dilation when $A_1, \dots , A_m$ are doubly commuting $2 \times 2$ scalar matrices. The general case, i.e. whether or not rational dilation succeeds on $\ov{\mathbb A}_r^m$, seems to be very nontrivial and challenging at this point.

\end{document}